\documentclass{amsart}
\usepackage{amsfonts}
\usepackage[alphabetic]{amsrefs} 
\usepackage{color, hyperref}

\newtheorem{theorem}{Theorem}[section]
\newtheorem{lemma}[theorem]{Lemma}
\theoremstyle{definition}

\theoremstyle{remark}
\newtheorem{remark}[theorem]{Remark}
\numberwithin{equation}{section}
\theoremstyle{plain}

\newtheorem{corollary}[theorem]{Corollary}

\begin{document}
\title {A Conformal Quasi Einstein\\ Characterization Of The Round Sphere}
\author{Ramesh Sharma}
\address{321 Maxcy Hall\\
Dept. of Math, University of New Haven\\
West Haven, CT 06516.}
\email{rsharma@newhaven.edu}
\urladdr{https://math.newhaven.edu/~rsharma}


\begin{abstract}
We extend the following result of Cochran ``A closed $m$-quasi Einstein manifold ($M,g,X$) with $m \ne -2$ has constant scalar curvature if and only if $X$ is Killing" covering the missing accidental case $m=-2$ and generalize it showing that $X$ is Killing if the integral of the Lie derivative of the scalar curvature along $X$ is non-positive. For a closed $m$-quasi Einstein manifold of dimension $n \ge 2$, if $X$ is conformal, then it is Killing; and in addition, if $M$ admits a non-Killing conformal vector field $V$, then it is globally isometric to a sphere and $V$ is gradient for $n > 2$. Finally, we derive an integral identity for a vector field on a closed Riemannian manifold, which provides a direct proof of the Bourguignon-Ezin  conservation identity.
\end{abstract}
\maketitle

\section{Introduction} By an $m$-quasi Einstein manifold, we mean a Riemannian manifold ($M,g$) that has a smooth vector field $X$, and satisfies the equation
\begin{equation}\label{1}
\frac{1}{2}\mathcal{L}_X g +Ric-\frac{1}{m}X^{*}\otimes X^{*}=\lambda g
\end{equation}
where $X^{*}$ denotes the 1-form dual to $X$, $\mathcal{L}_X$ denotes the Lie derivative operator along $X$, $Ric$ denotes the Ricci tensor of $g$, and $m \ne 0$ is a real number. The case $m=0$ may be interpreted as having $X=0$. The solution of (\ref{1}) is trivial when $X=0$ in which case $g$ becomes Einstein. Another case is the near horizon limit (Bahuaud-Gunasekaran-Kunduri-Woolgar \cite{BGKW2024}), which arises on vacuum stationary spacetimes that contain a so-called extreme black hole. For $m=2$, the solution of (\ref{1}) describes the geometry of a degenerate Killing horizon (black hole's event horizon) on which a Killing vector field that is timelike immediately outside the horizon becomes null. For $m=\pm \infty$, the equation (\ref{1}) reduces to a Ricci soliton. The case: $0<m< \infty$, has been studied considerably by some authors, e.g. Case-Shu-Wei \cite{C-S-W 2011} by taking $X$ as a gradient vector of a smooth function $f$ on $M$. For a positive integer $m$, ($M,g$) is $m$-quasi Einstein if and only if the warped product $M \times_{e^{-f/m}}F^{m}$ is Einstein, where $F^{m}$ is an $m$-dimensional Einstein manifold. We also note that the only closed 1-dimensional Riemannian manifold is $S^1$. The $g$-trace of equation (\ref{1}) is
\begin{equation}\label{2}
div X+R-\frac{1}{m}|X|^2=n\lambda.
\end{equation}

In \cite{Cochran2025}, Cochran proved the following result.\\

\noindent
\textbf{Theorem (Cochran)} \textit{Let ($M,g,X$) be a solution to the $m$-quasi Einstein equation (\ref{1}) such that ($M,g$) is closed ( i.e. compact and without boundary) and $m \ne -2$. Then ($M,g$) has constant scalar curvature if and only if $X$ is Killing}.\\

\noindent
So, the question arises whether the above mentioned result holds for the seemingly accidental and pathological case $m=-2$ as well. We show that this is indeed true and furthermore, obtain the following generalization of this result.

\begin{theorem} Let ($M,g,X$) be a solution to the $m$-quasi Einstein equation (\ref{1}) such that ($M,g$) is closed. If $X$ is Killing, then the scalar curvature $R$ is constant. If $\int_M (\mathcal{L}_X R) dv_g \le 0$ (in particular, $R$ is constant along $X$), then $X$ is Killing.
\end{theorem}
Here $dv_g$ denotes the volume element of ($M,g$). An immediate consequence of this result is the following corollary.
\begin{corollary}
For a closed $m$-quasi Einstein manifold, if the scalar curvature $R$ satisfies $\int_M (\mathcal{L}_X R) dv_g \le 0$, then $R$ is constant. 
\end{corollary}

The proof of Theorem 1.1 needs the following lemma.
\begin{lemma} For any vector field $X$ on a closed Riemannian manifold ($M,g$), $\int_M (\nabla_X |X|^2) dv_g=-\int_M |X|^2(div X)dv_g$ and if ($M,g$) is also $m$-quasi Einstein, then $\int_M (|X|^2 div X)dv_g=m\int_M [(div X)^2-\mathcal{L}_X R] dv_g$.
\end{lemma}

The first part of the lemma follows straightaway by integrating the identity: $div(|X|^2 X)=\nabla_X |X|^2 +|X|^2 div X$ over the closed $M$ and using divergence theorem. For the second part, we multiply (\ref{2}) with $div X$, and then integrate the resulting equation over $M$.\\

Now we turn our attention to proving Theorem 1.1.
\begin{proof}
First, let $X$ be Killing and hence $div X=0$.  We recall the following equations (equations (2.13) and (2.17) in \cite{BGKW2024}) for an $m$-quasi Einstein manifold:
\begin{eqnarray}\label{3}
0&=&\Delta |X|^2 -\frac{1}{2}(|S|^2 +|A|^2)-\nabla_X |X|^2 \nonumber\\
&+&2|X|^2(\lambda+\frac{1}{m}|X|^2)-\frac{4}{m}|X|^2div X.
\end{eqnarray}

\begin{eqnarray}\label{4}
&-&\Delta(div X)+(\frac{4}{m}+1)\nabla_X (div X)-2\lambda divX-\frac{1}{m}|X|^2divX \nonumber\\
&+&\frac{2}{m}(divX)^2=(\frac{1}{m}-\frac{1}{4})\Delta|X|^2-(\frac{3}{8}+\frac{1}{2m})|S|^2+(\frac{1}{2m}+\frac{1}{8})|A|^2\nonumber\\
&+& (\frac{3}{m}+\frac{1}{4})\nabla_X |X|^2-2(\frac{1}{m}+\frac{1}{4})|X|^2(\lambda+\frac{1}{m}|X|^2)
\end{eqnarray}
where $S$ and $A$ are tensors with components $S_{ij}=\nabla_i X_j +\nabla_j X_i=(\mathcal{L}_X g)_{ij}$ and $A_{ij}=\nabla_i X_j -\nabla_j X_i$, satisfying the properties: $|\nabla X|^2 =\frac{1}{4}(|S|^2+|A|^2)$, $(\nabla^i X^j)(\nabla_j X_i)=\frac{1}{4}(|S|^2-|A|^2)$. The assumption $S=0$ (which implies $div X=0$) reduces equations (\ref{3}) and (\ref{4}) to

\begin{equation}\label{5}
0=\Delta|X|^2-\frac{1}{2}|A|^2-\nabla_X |X|^2 +2|X|^2(\lambda-\frac{1}{2}|X|^2).
\end{equation}

\begin{eqnarray}\label{6}
&0&=(\frac{1}{m}-\frac{1}{4})\Delta |X|^2+(\frac{1}{2m}+\frac{1}{8}) |A|^2 +(\frac{3}{m}+\frac{1}{4})\nabla_X |X|^2 \nonumber\\
&-& 2(\frac{1}{m}+\frac{1}{4})|X|^2 (\lambda +\frac{1}{m}|X|^2).
\end{eqnarray}
Multiplying equation (\ref{5}) with $\frac{1}{m}+\frac{1}{4}$ and adding the resulting equation to equation (\ref{6}) gives

\begin{equation}\label{7}
0=\Delta|X|^2+\nabla_X |X|^2 .
\end{equation}
Now, using $div X=0$ in the trace equation (\ref{2}), and then differentiating along $X$ we get $\nabla_X |X|^2 =m \nabla_X R$. Using this in (\ref{7}) shows that
\begin{equation}\label{8}
\Delta |X|^2 +m\nabla_X R=0.
\end{equation}
Here we note that $X$ is Killing and hence preserves the scalar curvature, i.e, $\mathcal{L}_X R=0$, i.e. $\nabla_X R=0$. Use of this in (\ref{8}) gives $\Delta |X|^2 =0$. Hence, as $M$ is closed, by Bochner's lemma we conclude that $|X|$ is constant, and hence from (\ref{2}), $R$ is constant.\\

For the second part, we proceed as follows. The inner product of the fundamental equation (\ref{1}) with $X \otimes X$ and use of Lemma 1.3 gives
\begin{equation*}
\nabla_X |X|^2 +2Ric(X,X)-\frac{2}{m}|X|^4 -2\lambda |X|^2 =0.
\end{equation*}
Integrating it over the closed $M$ and using Lemma 1.3 we have
\begin{equation*}
\int_M [(divX) |X|^2) -2Ric(X,X)+2(\frac{|X|^4}{m}+\lambda |X|^2)]dv_g =0.
\end{equation*}
At this point, we recall the following integral formula for a smooth vector field $X$ on a closed Riemannian manifold ($M,g$) [equation (1.11), p. 41 of \cite{Yano1970}]:

\begin{equation*}
\int_M [Ric(X,X)+\frac{1}{2}|\mathcal{L}_X g|^2-|\nabla X|^2 -(div X)^2]dv_g =0.
\end{equation*}
Eliminating $\int_M Ric(X,X)dv_g$ between the two preceding equations provides
\begin{equation}\label{9}
\int_M [(divX) |X|^2) +2(\frac{|X|^4}{m}+\lambda |X|^2)+|\mathcal{L}_X g|^2-2|\nabla X|^2-2(div X)^2]dv_g =0,
\end{equation}
Using the second part of Lemma 1.3, in (\ref{9}) we find that
\begin{eqnarray}\label{10}
\int_M[(2-m)(div X)^2&+& m\mathcal{L}_X R -\frac{2}{m}(|X|^4 +m\lambda|X|^2)\nonumber\\
&-&|\mathcal{L}_X g|^2 +2|\nabla X|^2]dv_g=0.
\end{eqnarray}
Next, we recall the following equation for an $m$-quasi Einstein metric (Bahuaud, Gunasekaran, Kunduri and Woolgar \cite{BGKW2024}):
\begin{eqnarray*}
0&=&\Delta |X|^2-\frac{1}{2}(|S|^2+|A|^2)-\nabla_X |X|^2\nonumber\\
&+&2|X|^2(\lambda +\frac{1}{m}|X|^2)-\frac{4}{m}|X|^2 div X.
\end{eqnarray*}
Integrating it, adding the resulting equation to (\ref{10}) and using the property \cite{BGKW2024}: $|\nabla X|^2 =\frac{1}{4}(|S|^2+|A|^2)$ along with Lemma 1.3 we obtain 
\begin{equation*}
\int_M [|\mathcal{L}_X g|^2 + 2(div X)^2]dv_g=4\int_M (\mathcal{L}_X R)dv_g.
\end{equation*}
Hence, if $\int_M (\mathcal{L}_X R)dv_g \le 0$, then the above integral equality implies that $X$ is Killing. This completes the proof.
\end{proof}
\begin{remark} The integral $\int_M (\mathcal{L}_X R)dv_g$ involved in Theorem 1.1 plays an important role in conformal geometry, and vanishes for a conformal vector field $X$ on a closed Riemannian manifold (Bourgignon-Ezin identity \cite{Bourgignon-Ezin 1987}).
\end{remark}
\begin{remark} First, we noticed during the proof of Theorem 1.1 that $|X|$ is constant on $M$. So, if $X$ is non-zero then it is nowhere zero and hence $M$ has its Euler characteristic equal to zero. Secondly, it was pointed out in \cite{Cochran2025} that the Ricci tensor of a closed $m$-quasi manifold ($M,g,X$) with constant scalar curvature (hence $X$ Killing) need not be parallel. Here, we show that the Ricci tensor of such a manifold is cyclically parallel as follows. We first observe that equation (\ref{1}) becomes $R_{jk} =\lambda g_{jk}+\frac{1}{m}X_j X_k$, when $X$ is Killing. As $|X|$ is constant, we have $g(\nabla_Y X,X)=0$ for any $Y \in \mathfrak{X}(M)$, and since $X$ is Kiiling, we get $g(\nabla_X X,Y)=0$ which implies $\nabla_X X=0$. Thus $Ric$ is parallel along the geodesics determined by $X$. On the other hand, using the aforementioned expression for $R_{ij}$, we compute $\nabla_i R_{jk}= \frac{1}{m}[(\nabla_i X_j)X_k+X_j \nabla_i X_k]$. Permuting $i,j,k$ cyclically twice, adding the resulting two equations to the foregoing equation we find that $\nabla_i R_{jk}+\nabla_J R_{ki}+\nabla_k R_{ij}=0$, i.e $Ric$ is cyclically parallel. Riemannian metrics with cyclically parallel Ricci tensor have constant scalar curvature and were studied by Gray \cite{Gray 1978}) who constructed such metrics on the unit sphere $S^3$ with Ricci tensors cyclically parallel, but not parallel. Another example of a manifold whose Ricci tensor is cyclically parallel but not parallel, given in \cite{Gray 1978}, is $O(4)/O(2)$ with a bi-invariant metric.
\end{remark}

In the rest of this paper, for a closed $m$-quasi Einstein manifold of dimension $n \ge 2$, if $X$ is conformal, then we show that it is Killing; and in addition, if $M$ admits a non-Killing conformal vector field $V$, then it is globally isometric to a sphere and $V$ is necessarily gradient (up to the addition of a Killing vector field). This is the main result of this paper. Along the way, an integral formula is derived, giving a direct proof of the result `` $div X=0$ on the closed $m$-quasi Einstein ($M,g,X$) implies that $X$ is Killing" proved by Cochran  for any $m$ using a special construction. Finally, we provide a direct derivation of the Bourguignon-Ezin  conservation identity \cite{Bourgignon-Ezin 1987}: $\int_M (\mathcal{L}_V R) dv_g = 0$, for a conformal vector field $V$ on a closed Riemannian manifold $M^n$ ($n>2$), and obtain a generalization of this identity for any smooth vector field on a closed Riemannian manifold for $n>2$.

\section{$m$-quasi Einstein Theoretic Characterization Of A sphere}
We first consider the case when $X$ is conformal on the closed $m$-quasi Einstein manifold and prove the following result.
\begin{theorem} Let ($M,g,X$) be a solution to the $m$-quasi Einstein equation such that ($M,g$) is closed and of dimension $n \ge 2$. If $X$ is conformal, then it is Killing.
\end{theorem}

\begin{proof}
The $m$-quasi Einstein equation (\ref{1}) and its trace (\ref{2}) can be written in the component form as
\begin{equation*}
\nabla_j X^i +\nabla^i X_j +2R^i _j-\frac{2}{m}X^i X_j=2\lambda \delta ^i _j ,
\end{equation*}
\begin{equation*}
\nabla_i X^i +R-\frac{1}{m}|X|^2=n\lambda .
\end{equation*}
Taking their covariant derivatives and using the twice contracted Bianchi's second identity: $\nabla_i R^i _j=\frac{1}{2}\nabla_j R$, we have
\begin{equation*}
\nabla_i \nabla_j X^i +\nabla_i \nabla^i X_j+\nabla_j R-\frac{2}{m}[(div X)X_j +(\nabla_X X)_j]=0,
\end{equation*}
\begin{equation*}
\nabla_j \nabla_i X^i +\nabla_j R-\frac{1}{m}\nabla_j |X|^2=0.
\end{equation*}
Now subtracting the second equation from the first, provides
\begin{equation}\label{11}
-\square X=\frac{1}{m}[2(div X)X+2\nabla_X X- \nabla|X|^2]
\end{equation}
where the Yano operator $\square: \mathfrak{X}(M) \rightarrow \mathfrak{X}(M)$ is defined by (p. 40 in \cite{Yano 1970})
\begin{equation*}
(\square X)^i =-(\nabla^j \nabla_j X^i +R^i _j X^j).
\end{equation*}
Next, taking the inner product of (\ref{11}) with $X$, noting that $\nabla_X |X|^2=2g(\nabla_X X,X)$ and subsequently integrating over the closed $M$ yields that
\begin{equation}\label{12}
-\int_M g(\square X,X)dv_g=\frac{2}{m}\int_M (div X)|X|^2 dv_g.
\end{equation}
which holds for any closed $m$-quasi Einstein manifold. At this point, we recall the following integral formula for a vector field on a closed Riemannian manifold ($M,g$) [equation (1.14), p. 41 in \cite{Yano 1970}]:
\begin{equation}\label{13}
\int_M[g(\square X-\frac{n-2}{n}\nabla(div X),X)-\frac{1}{2}|\mathcal{L}_X g-\frac{2div X}{n}g|^2]dv_g=0.
\end{equation}
The use of (\ref{12}) and the identity: $div((div X)X)=\nabla_X  div X+(div X)^2$ in equation (\ref{13}) provides the formula:
\begin{equation}\label{14}
\int_M [\frac{2}{m}(div X)|X|^2 -\frac{n-2}{n}(div X)^2+\frac{1}{2}|\mathcal{L}_X g-\frac{2div X}{n}g|^2]dv_g=0
\end{equation}
for a closed $m$-quasi Einstein manifold. Using the second part of Lemma 1.3 in equation (\ref{14}) we obtain
\begin{equation*}
\int_M [\frac{n+2}{2}(div X)^2 -2\mathcal{L}_X R +\frac{1}{2}|\mathcal{L}_X g-\frac{2div X}{n}g|^2]dv_g =0.
\end{equation*}
So, if $X$ is conformal, then $\mathcal{L}_X g=\frac{2div X}{n}g$. In the above equation, we use Bourguignon-Ezin  conservation identity \cite{Bourgignon-Ezin 1987}: $\int_M (\mathcal{L}_X R) dv_g = 0$, for a conformal vector field $X$ on a closed Riemannian manifold, and obtain $div X=0$. Consequently, as pointed out in the remark below, it follows that $X$ is Killing. This completes the proof.
\end{proof}

\begin{remark} As the integral equation (\ref{14}) holds for a closed $m$-quasi Einstein manifold, we immediately recover the following result: ``If $div X=0$ on a closed $m$-quasi manifold, then $X$ is Killing" proved in \cite{BGKW2024} for $m \ne -2$, and in \cite{Cochran2025} for any $m$.
\end{remark}

We noticed from the previous theorem that demanding $X$ for a closed $m$-quasi Einstein manifold, to be conformal, forces it to become Killing. Thus, one may consider another vector field (other than $X$) and check the impact on the closed $m$-quasi Einstein manifold. Motivated by this consideration, we establish the following characterization of a sphere, as the main result of this paper.

\begin{theorem}Let ($M,g,X$) be an $n$-dimensional ($n \ge 2$) closed $m$-quasi Einstein manifold such that $X$ is Killing. If it admits a non-Killing conformal vector field $V$, then it is globally isometric to a sphere. Also, in dimensions higher than 2, $V$ is gradient (up to the addition of a Killing vector field).
\end{theorem}
\begin{proof} Let $V$ be a non-Killing conformal vector field on ($M,g$). Then
\begin{equation*}
\mathcal{L}_V g=2\sigma g
\end{equation*}
for a non-constant smooth function $\sigma$ on $M$. We recall the following integrability condition for the conformal vector field $V$ \cite{Yano 1970}:
\begin{equation}\label{15}
\mathcal{L}_V R=-2\sigma R-2(n-1)\Delta \sigma.
\end{equation}
First, we consider the case: $n=2$. As $X$ is Killing, Theorem 1.1 asserts that $R$ is constant. Hence (\ref{15}) reduces to $\Delta \sigma =-R \sigma$. Hence $R$ is positive. Therefore $M$ has positive constant curvature. Applying the following  result of Bishop-Goldberg \cite{B-G 1966} ``A compact Riemannian manifold of constant positive curvature is globally isometric to a round sphere, we conclude that ($M,g$) is globally isometric to a 2-sphere.\\

For case $n > 2$, as $X$ is Killing, by Theorem 1.1, $R$ is constant. Using this in the integral of (\ref{2}) shows that $|X|$ is constant. Thus equation (\ref{15}) reduces to $\Delta \sigma =-\frac{R}{n-1}\sigma$. This shows that $R >0$. As $X$ is Killing, the $m$-quasi Einstein equation (\ref{1}) reduces to the form:
\begin{equation}\label{16}
Ric=\frac{1}{m}X^{*} \otimes X^{*}+\lambda g.
\end{equation}
Using this we compute 
\begin{equation}\label{17}
|Ric|^2 =\frac{1}{m^2}[|X|^4 +nm^2 \lambda^2 +2\lambda m|X|^2].
\end{equation}
As $|X|$ is constant, the equation (\ref{17}) shows that $|Ric|$ is constant on $M$. Applying the following result of Lichnerowicz \cite{Lichnerowicz 1964}: ``If a closed Riemannian manifold ($M,g$) of dimension $>2$ with constant scalar curvature and constant norm of the Ricci tensor, admits an infinitesimal non-isometric conformal transformation, then ($M,g$) is globally isometric to a round sphere" we conclude that ($M,g$) is globally isometric to a round sphere. So, $M$ is Einstein and hence equation (\ref{16}) reduces to 
\begin{equation}\label{18}
(\frac{R}{n}-\lambda)g=\frac{1}{m}X^{*}X^{*}.
\end{equation}
Contracting it gives $R-n\lambda=\frac{1}{m}|X|^2$. Next, operating (\ref{18}) on the pair ($X,$X) provides $(\frac{R}{n}-\lambda)|X|^2=\frac{1}{m}|X|^4$. The preceding two equations imply that $X=0$. Now we recall the following integrability condition for a conformal vector field $V$ \cite{Yano 1970}:
\begin{equation*}
\mathcal{L}_V Ric=(2-n)\nabla \nabla \sigma -(\Delta \sigma)g.
\end{equation*}
Using this in the Einstein condition: $Ric =\frac{R}{n}g$,  combining it  with (\ref{15}) and taking into account that $n>2$, shows that $\nabla \nabla \sigma=\frac{R}{n(1-n)}\sigma g$, i.e. $\mathcal{L}_{\nabla \sigma}g=2\frac{R}{n(1-n)}\sigma g$. Combining this with the conformal equation $\mathcal{L}_V g=2\sigma g$ immediately yields that $\mathcal{L}_{V+\frac{n(n-1)}{R}\nabla \sigma}g=0$ which proves (as $R$ is a positive constant) that $V$ is gradient up to the addition of a Killing vector field, completing the proof. 
\end{proof}
\begin{remark} Theorem 2.3 is a generalization of the following important result of Yano-Nagano \cite{Yano-Nagano 1959} ``If a closed Riemannian manifold $M$ of dimension $n>2$ admits an infinitesimal non-isometric conformal transformation, then $M$ is globally isometric to a sphere" for which $X=0$.
\end{remark}

\section{A generalization of Bourguignon-Ezin  conservation identity} A conformal vector field $V$ on an $n$-dimensional Riemannian manifold ($M,g$) is defined by
\begin{equation*}
\mathcal{L}_V g=2\sigma g
\end{equation*}
for a smooth function $\sigma$ on $M$. The Bourguignon-Ezin  conservation identity: $\int_M (\mathcal{L}_V R) dvol_g = 0$ holds for a conformal vector field $V$ on a closed Riemannian manifold of dimension $\ge 2$. Its derivation for $n>2$ uses the integral of following integrability condition \cite{Yano 1970} for a conformal vector field $V$ :
\begin{equation*}
\mathcal{L}_V R=-2\sigma R+2(1-n)\Delta \sigma.
\end{equation*}
along with the identity: $div (R V)=\mathcal{L}_V R +Rdiv(V)$ and the trace $div(V)=n\sigma$ of the conformal equation. We present a direct derivation of the Bourguignon-Ezin  conservation identity (without using the above mentioned integrability condition) for $n > 2$ as follows. Suppose $T$ is a seond order symmetric tensor on $M$. Then we compute
\begin{eqnarray*}
&&<\mathcal{L}_V g, T>=(\nabla_i V_j +\nabla_j V_i)T^{ij} \nonumber\\
&=&2(\nabla_i V_j) T^{ij})=2[\nabla_i (X_j T^{ij})-X_j \nabla_i T^{ij}]
\end{eqnarray*}
which can be expressed as $<\mathcal{L}_V g, T>=2[div(i_V T)-(div T)X]$ that was first derived by Griffin in \cite{Griffin2021}. Integrating it over a closed Riemannian manifold ($M,g$) and replacing $T$ with the trace free part $\mathring{Ric}=Ric -\frac{R}{n}g$ of the Ricci tensor and using the twice contracted Bianchi second identity: $div(Ric)=\frac{1}{2}\nabla R$ provides the following generalization of Bourguignon-Ezin  conservation identity:
\begin{equation*}
\int_M [<\mathring{\mathcal{L}_V g},\mathring{Ric}>+\frac{n-2}{n}\mathcal{L}_V R]dv_g =0
\end{equation*}
for any smooth vector field $V$ (not necessarily conformal) on a closed Riemannian manifold, where $\mathring{\mathcal{L}_V g}$ denotes the trace free part of $\mathcal{L}_V g$. In particular, when $V$ is conformal, i.e. $\mathring{\mathcal{L}_V g}=0$, the above formula recovers the Bourguignon-Ezin identity for $n > 2$.

\section{Acknowledgement} I express my sincere gratitude to Alex Colling who pointed out some errors in the earlier version.\\

\noindent
\textbf{Competing Interests}: The author has no relevant financial or non-financial interests to disclose.\\

\noindent
\textbf{Data Availability}: Data sharing is not applicable to this article as no datasets were generated or analysed during the current study.\\

\bibliographystyle{alpha}

\end{document}